\documentclass{siamltex}

\usepackage{amsmath}
\usepackage{amssymb}
\newcommand{\Z}{{\mathbb{Z}}}

\newcommand{\N}{{\mathbb{N}}}
\newcommand{\I}{{\mathbb{I}}}
\newcommand{\R}{{\mathbb{R}}}
\newcommand{\E}{{\mathbb{\,E\,}}}

\newcommand{\eps}{\varepsilon}

\newcommand{\prob}[1]{\mathbb{P}\left\{#1\right\}}
\newcommand{\Ind}{\mathbb{I}}
\newcommand{\half}{ {\textstyle \frac 1 2 }}
\title{Optimal Probability Inequalities for Random Walks Related to Problems in Extremal Combinatorics}
\author{D.\,Dzindzalieta \footnotemark[2]\ \footnotemark[5] \and  T.\,Ju\v{s}kevi\v{c}ius \footnotemark[3] \and M.\,\v{S}ileikis \footnotemark[4]\ \footnotemark[5] }
\begin{document}
\maketitle

\renewcommand{\thefootnote}{\fnsymbol{footnote}}

\footnotetext[2]{Vilnius University Institute of Mathematics and Informatics, Akademijos 4, LT-08663, Vilnius, Lithuania (dainiusda@gmail.com)}
\footnotetext[3]{Department of Mathematical Sciences, University of Memphis, Memphis, TN 38152-3240, USA (tomas.juskevicius@gmail.com)}
\footnotetext[4]{Department of Discrete Mathematics, Adam Mickiewicz University, Umultowska 87, 61-614 Pozna\'{n}, Poland (matas.sileikis@gmail.com)}
\footnotetext[5]{This research was funded by a grant (No. MIP-47/2010) from the Research Council of Lithuania}
\begin{center}
Dedicated to the memory of V.\,Bentkus
\end{center}
\vspace{8pt}

\begin{abstract}
  Let $S_n=X_1+\cdots+X_n$ be a sum of independent symmetric random variables such that $\left|X_{i}\right|\leq 1$. Denote by $W_n=\varepsilon_{1}+\cdots+\varepsilon_{n}$ a sum of independent random variables such that $\prob{\eps_i = \pm 1} = 1/2$.
We prove that
\begin{equation*}
  \mathbb{P}\left\{S_{n} \in A\right\} \leq \mathbb{P}\left\{cW_k \in A\right\},
\end{equation*}
where $A$ is either an interval of the form $\left[x, \infty \right)$ or just a single point. The inequality is exact and the optimal values of $c$ and $k$ are given explicitly. It improves Kwapie\'{n}'s inequality in the case of the Rademacher series. We also provide a new and very short proof of the Littlewood-Offord problem without using Sperner's Theorem. Finally, an extension to odd Lipschitz functions is given.
 
\end{abstract}



\section{Introduction}

Let $S_{n}=X_1+\cdots+X_n$ be a sum of independent random variables $X_i$ such that
\begin{equation}
  \label{salygos}
\left|X_i\right| \leq 1 \quad \text{and} \quad \E X_i=0.
\end{equation}
Let $W_n = \varepsilon_{1}+\cdots+\varepsilon_n$ be the sum of independent {Rademacher random variables}, i.e., such that $\prob{\eps_i = \pm 1} = 1/2$. We will refer to $W_n$ as a simple random walk with $n$ steps. 

By a classical result of Hoeffding \cite{hoeffding} we have the estimate
\begin{equation}
\mathbb{P}\left\{S_n \geq x\right\} \leq \exp\left\{-x^2/2n\right\}, \quad x\in \R.
\label{h}
\end{equation} 
If we take $S_n=W_n$ on the left-hand side of \eqref{h}, then in view of the Central Limit Theorem we can infer that the exponential function on the right-hand side is the minimal one. Yet a certain factor of order $x^{-1}$ is missing, since $\Phi(x)\approx (\sqrt{2\pi}x)^{-1}\exp\left\{-x^2/2\right\}$ for large $x$.

Furthermore, it is possible to show that the random variable $S_n$ is sub-gaussian in the sense that 
\begin{equation*}
\mathbb{P}\left\{S_n \geq x\right\} \leq  c\, \mathbb{P}\left\{\sqrt{n}Z \geq x\right\}, \quad x\in \R,
\end{equation*} 
where $Z$ is the standard normal random variable and $c$ is some explicit positive constant (see, for instance, \cite{b07}).

Perhaps the best upper bound for $\mathbb{P}\left\{S_n \geq x\right\}$ was given by Bentkus \cite{b01}, which for integer $x$ is optimal for martingales with differences $X_i$ satisfying \eqref{salygos}. 

Although there are numerous improvements of the Hoeffding inequality, to our knowledge there are no examples where the exact bound for the tail probability is found. In this paper we give an optimal bound for the tail probability $\mathbb{P}\left\{S_n \geq x\right\}$ under the additional assumption of symmetry. 

We henceforth reserve the notation $S_n$ and $W_n$ for random walks with symmetric steps satisfying \eqref{salygos} and a simple random walk with $n$ steps respectively. 
\begin{theorem}\label{thm1}
For $x>0$ we have
\begin{equation}\label{eq1}
\mathbb{P}\left\{S_n \geq x\right\} \leq \left\{ \begin{gathered}
\mathbb{P}\left\{W_n \geq x\right\} \quad \quad \text{if}\,\,\, \left\lceil x\right\rceil+n \in 2\Z, \hfill \\
\mathbb{P}\left\{W_{n-1} \geq x\right\}\quad \text{if}\,\,\, \left\lceil x\right\rceil+n \in 2\Z+1.  \hfill \\
\end{gathered} \right.
\end{equation}
\end{theorem}

The latter inequality can be interpreted by saying that among bounded random walks the simple random walk is the most stochastic. 

Kwapie\'{n} proved (see \cite{sztencel}) that for arbitrary i.i.d. symmetric random variables $X_i$ and real numbers $a_i$ with absolute value less than $1$ we have
\begin{equation*}
\mathbb{P}\left\{a_1X_1+\ldots+a_nX_n \geq x\right\} \leq 2\, \mathbb{P}\left\{X_1+\ldots+X_n \geq x\right\}, \quad x>0.
\end{equation*}  
The case $n=2$ with $X_i = \eps_i$ shows that the constant $2$ cannot be improved.
 
Theorem \ref{thm1} improves Kwapie\'{n}'s inequality for Rademacher sequences. We believe that using the inequality in \eqref{eq1} with some conditioning arguments leads to better estimates for arbitrary symmetric random variables $X_i$ under the assumptions of Kwapie\'{n}'s inequality, but we will not go into these details in this paper. 

We also consider the problem of finding the quantity
\begin{equation*}
\sup_{S_n} \mathbb{P}\left\{S_n = x\right\},
\end{equation*}
which can be viewed as a non-uniform bound for the concentration of the random walk $S_n$ at a point.
\begin{theorem}\label{thm2}
For $x>0$ and $k = \lceil x \rceil$ we have
\begin{equation}
  \label{taskas}
\mathbb{P}\left\{S_n = x\right\}\leq \mathbb{P}\left\{W_m = k\right\},
\end{equation}
where 
\begin{equation*}
m=\left\{\begin{gathered}
\min\left\{n,k^2 \right\},\;\;\,  \qquad\text{if}\,\,\, n+k \in 2\Z, \hfill \\
\min\left\{n-1,k^2\right\},\quad \text{if}\,\,\, n+k \in 2\Z+1.\\
\end{gathered} \right.
\end{equation*}
\end{theorem}
\begin{flushleft}
  Equality in \eqref{taskas} is attained for $S_n=\frac{x}{k}\,W_m$.\\
\end{flushleft}
We provide two different proofs for both inequalities. The first approach is based on induction on the number of random variables (\S \ref{sec2}). To prove Theorem \ref{thm2} we also need the solution of the Littlewood-Offord problem. 

\begin{theorem}
  \label{LO}
Let $a_{1},\ldots,a_{n}$ be real numbers such that $\left|a_{i}\right|\geq 1$. We have 
\begin{equation*}
\max_{x\in \R}\mathbb{P}\left\{S_{n}\in (x-k,x+k]\right\} \leq \mathbb{P}\left\{W_{n}\in (-k,k]\right\}.
\end{equation*}
\end{theorem}
\noindent That is, the number of the choices of signs for which $S_{n}$ lies in an interval of length $2k$ does not exceed the sum of $k$ largest binomial coefficients in $n$.

Theorem \ref{LO} was first proved by Erd\H{o}s \cite{erdos} using Sperner's Theorem. We give a very short solution which seems to be shorter than the original proof by Erd\H{o}s. We only use induction on $n$ and do not use Sperner's Theorem.

Surprisingly, Theorems \ref{thm1} and \ref{thm2} can also be proved by applying results from extremal combinatorics (\S \ref{sec3}). Namely, we use the bounds for the size of intersecting families of sets (hypergraphs) by Katona \cite{katona} and Milner \cite{milner}.  

 Using a strengthening of Katona's result by Kleitman \cite{kleitman}, we extend Theorem \ref{thm1} to odd 1-Lipschitz functions rather than just sums of the random variables $X_i$ (\S \ref{sec4}). It is important to note that the bound of Theorem \ref{thm1} cannot be true for all Lipschitz functions since the extremal case is not provided by odd functions (for the description of the extremal Lipschitz functions defined on general probability metric spaces see Dzindzalieta \cite{dzindzalieta}).

\section{Proofs by induction on dimension}\label{sec2}
We will first show that it is enough to prove Theorems \ref{thm1} and \ref{thm2} in case when $S_n$ is a linear combination of independent Rademacher random variables $\varepsilon_{i}$ with coefficients $\left|a_{i}\right|\leq 1$. 
 \begin{lemma}\label{reduction}
 Let $g : \R^n \to \R$ be a bounded  measurable function. Then we have 
  $$\sup_{X_1, \dots, X_n}\E g(X_1, \dots, X_n) = \sup_{a_1, \dots, a_n} \E g(a_1 \eps_1, \dots a_n \eps_n),$$
  where the supremum on the left-hand side is taken over symmetric independent random variables $X_1, \dots, X_n$ such that $|X_i| \leq 1$ and the supremum on the right-hand side is taken over numbers  $-1 \leq a_1, \dots, a_n \leq 1$.
\end{lemma}

{\em Proof}. Define $S = \sup_{a_1, \dots, a_n} \E g(a_1 \eps_1, \dots a_n \eps_n)$. Clearly 
  $$S \leq \sup_{X_1, \dots, X_n}\E g(X_1, \dots, X_n).$$
  By symmetry of $X_1, \dots, X_n$, we have
  $$ \E g(X_1, \dots, X_n) = \E g(X_1\eps_1, \dots, X_n\eps_n).$$
  Therefore
  \begin{align*}
    \E g(X_1, \dots, X_n) &= \E \E[ g(X_1\eps_1, \dots, X_n\eps_n) \,|\, X_1, \dots, X_n] \leq \E S=S. \qquad\endproof
  \end{align*} 

Thus, in view of Lemma \ref{reduction} we will henceforth write $S_n$ for the sum $a_1 \eps_1 + \dots + a_n \eps_n$ instead of a sum of arbitrary symmetric random variables $X_i$.

{\em Proof of Theorem \ref{thm1}}. First note that the inequality is true for $x \in (0,1]$ and all $n$. This is due to the fact that $\mathbb{P}\left\{S_n \geq x\right\}\leq 1/2$ by symmetry of $S_n$ and for all $n$ the right-hand side of the inequality is given by the tail of an odd number of random signs, which is exactly $1/2$. 
  We can also assume that the largest coefficient $a_{i}=1$ as otherwise if we scale the sum by $a_{i}$ then the tail of the this new sum would be at least as large as the former. We  will thus assume, without loss of generality, that $0 \leq a_{1}\leq a_{2}\leq \ldots \leq a_{n}=1$. Define a function $\I(x,n)$ to be $1$ if $\lceil x \rceil + n$ is even, and zero otherwise. Then we can rewrite the right-hand side of \eqref{eq1} as
  $$\prob{W_{n-1} + \eps_n\I(x, n)  \geq x},$$
making an agreement $\eps_0 \equiv 0$.

  For $x>1$ 
  we argue by induction on $n$. Case $n = 0$ is trivial. Observing that $\I(x-1,n) = \I(x+1,n) = \I(x, n+1)$ we have
  \begin{align*}
    \prob{S_{n+1}\geq x} &= \half \prob{S_{n}\geq x-1} + \half\prob{S_{n}\geq x+1 }\\
    &\leq \half \prob{ W_{n-1}+\varepsilon_{n}\Ind(x-1,n) \geq x-1 } \\
    &+ \half\prob{ W_{n-1}+\varepsilon_{n}\I(x + 1, n) \geq x+1 }\\
    &=\prob{ W_{n}+\varepsilon_{n+1}\I(x,n+1)\geq x }. \qquad\endproof
  \end{align*}

  {\em Proof of Theorem \ref{LO}}. We can assume that $a_{1}\geq a_{2}\geq \ldots \geq a_{n} \geq 1$. Without loss of generality we can also take $a_{n}=1$. This is because
  \begin{align*}
    \mathbb{P}\left\{S_{n}\in (x-k,x+k]\right\} &\leq \mathbb{P}\left\{S_{n}/a_{n}\in (x-k,x+k]/a_{n}\right\} \\ 
    &\leq \max_{x\in \R}\mathbb{P}\left\{S_{n}/a_{n}\in (x-k,x+k]\right\}.
  \end{align*}

  The claim is trivial for $n=0$. Let us assume that we have proved the statement for $1,2,...,n-1$. Then
  \begin{align*}
    \mathbb{P}&\left\{S_{n}\in (x-k,x+k]\right\}\\
    =& \half\mathbb{P}\left\{S_{n-1}\in (x-k-1,x+k-1]\right\}+\half\mathbb{P}\left\{S_{n-1}\in (x-k+1,x+k+1]\right\}\\
    =&\half\mathbb{P}\left\{S_{n-1}\in (x-k-1,x+k+1]\right\}+\half\mathbb{P}\left\{S_{n-1}\in (x-k+1,x+k-1]\right\} \\
    \leq& \half\mathbb{P}\left\{W_{n-1}\in (-k-1,k+1]\right\}+\half\mathbb{P}\left\{W_{n-1}\in (-k+1,k-1]\right\}\\
    =& \half\mathbb{P}\left\{W_{n-1}\in (-k-1,k-1]\right\}+\half\mathbb{P}\left\{W_{n-1}\in (-k+1,k+1]\right\}\\
    =&\mathbb{P}\left\{W_{n}\in (-k,k]\right\}. \qquad \endproof
  \end{align*}

\begin{flushleft}
Note that we rearranged the intervals after the second equality so as to have two intervals of different lengths and this makes the proof work.
\end{flushleft}

Before proving Theorem \ref{thm2}, we will obtain an upper bound for $\prob{S_n = x}$ under an additional condition that all $a_i$ are nonzero.

\begin{lemma}\label{apelsinas}
  Let $x >0$, $k= \lceil x \rceil$. Suppose that $0 < a_1 \leq \dots \leq a_n \leq 1$. Then
  \begin{equation}\label{nialygybe}
    \prob{S_n = x} \leq 
      \begin{cases}
        \prob{W_n = k}, \quad &\text{\rm if} \quad n+k \in 2\Z, \\
        \prob{W_{n-1} = k}, \quad &\text{\rm if} \quad n+k \in 2\Z + 1.
      \end{cases}
  \end{equation}
\end{lemma}
{\em Proof}. We first prove the lemma for $x \in (0, 1]$ and any $n$.
  By Theorem \ref{LO} we have
  \begin{equation}\label{LO1}
    \prob{S_n = x} \leq 2^{-n} \binom{n}{\lceil n / 2 \rceil}.
  \end{equation}
  On the other hand, if $x \in (0, 1]$, then $k = 1$ and
  \begin{equation*}
    2^{-n} \binom{n}{\lceil n / 2 \rceil} = \left \{
      \begin{gathered}
        2^{-n} \binom{n}{(n + 1) / 2} = \prob{W_n = 1}, \quad \,\, \text{if} \quad n + 1 \in 2\Z, \hfill \\
	2^{-n} \binom{n}{n / 2} = \prob{W_{n-1} = 1}, \quad \qquad \text{if} \quad n + 1 \in 2\Z + 1, \hfill
      \end{gathered} \right.
  \end{equation*}
  where the second equality follows by Pascal's identity:
  $$2^{-n} \binom{n}{n/2} = 2^{-n} \left[ \binom{n-1}{n/2} + \binom{n-1}{n/2 - 1} \right]= 2^{1-n} \binom{n-1}{n/2} = \prob{W_{n-1} = 1}.$$
  Let $\N = \{1, 2, \dots \}$ stand for the set of positive integers. Let us write $B_n(x)$ for the right-hand side of \eqref{nialygybe}. Note that it has the following properties:
  \begin{align}
    &x \mapsto B_n(x) \text{ is non-increasing}; \label{pirma} \\
    &x \mapsto B_n(x) \text{ is constant on each of the intervals } (k-1, k], \quad k \in \N \label{antra}; \\
    &B_n (k) = \half B_{n-1}(k-1) + \half B_{n-1}(k+1), \quad \text{if } k = 2, 3, \dots. \label{trecia}
  \end{align}

  We proceed by induction on $n$. The case $n=1$ is trivial. To prove the induction step for $n \geq 2$, we consider two cases: (i) $x = k \in \N$; (ii) $k - 1 < x < k \in \N$.
  
  Case (i). For $k = 1$ the lemma has been proved, so we assume that $k \geq 2$. By the inductional hypothesis we have 
  \begin{align}
    \prob{S_n = k} &= \half \prob{S_{n-1} = k - a_n} + \half \prob{S_{n-1} = k + a_n} \nonumber \\
                   &\leq \half B_{n-1}(k - a_n) + \half B_{n-1}(k + a_n). \label{aaa}
  \end{align}
  By \eqref{pirma} we have
  \begin{equation}\label{bbb}
    B_{n-1}(k - a_n) \leq B_{n-1}(k - 1),
  \end{equation}
  and by \eqref{antra} we have
  \begin{equation}\label{ccc}
    B_{n-1}(k + a_n	) = B_{n-1}(k + 1).
  \end{equation}
  Combining \eqref{aaa}, \eqref{bbb}, \eqref{ccc}, and \eqref{trecia}, we obtain
  \begin{equation}\label{ddd}
    \prob{S_n = k} \leq B_n(k).  
  \end{equation}

  Case (ii). For $x \in (0, 1]$ Lemma has been proved, so we assume $k \geq 2$. Consider two cases: (iii) $ x/ a_n \geq k$; (iv) $x/ a_n < k$.

  Case (iii). Define $S_n' = a_1' \eps_1 + \dots + a_n' \eps_n$, where $a_i' = {k a_i}/{x}$, so that $S_n' = \frac k x S_n$. Recall that $a_n = \max_i a_i$, by the hypothesis of Lemma. Then $a_i' \leq ka_n/x$ and the assumption $x/a_n \geq k$ imply that $0 < a_1', \dots, a_n' \leq 1$. Therefore, by \eqref{ddd} and \eqref{antra} we have
  \begin{equation*}\label{eee}
    \prob{S_n = x} = \prob{S_n' = k} \leq B_n(k) = B_n(x).
  \end{equation*}

  Case (iv). Without loss of generality, we can assume that $a_n = 1$, since 
$$\prob{S_n = x} = \prob{\frac{a_1}{a_n}\eps_1 + \dots + \frac{a_n}{a_n}\eps_n = \frac{x}{a_n}}$$
  and $k - 1 < x/a_n < k$, by the assumption of the present case. Sequentially applying the induction hypothesis, \eqref{antra}, \eqref{trecia}, and again \eqref{antra}, we get
  \begin{align*}
    \prob{S_n = x} &= \half \prob{S_{n-1} = x - 1} + \half \prob{S_{n-1} = x + 1} \\
		   &\leq \half B_{n-1}(x - 1) + \half B_{n-1}(x + 1) \\
		   &= \half B_{n-1}(k - 1) + \half B_{n-1}(k + 1) \\
		   &= B_n(k) = B_n(x). \qquad\endproof
  \end{align*}

  {\em Proof of Theorem \ref{thm2}}. Writing $B_n(k)$ for the right-hand side of \eqref{nialygybe}, we have, by Lemma \ref{apelsinas}, that 
\[ \prob{S_n = x} \leq  \max_{ j = k}^n B_j(k).\]
If $j + k \in 2\Z$, then $B_j(k) = \prob{W_j = k} = B_{j+1}(k)$ and therefore
\begin{equation}\label{bananas}
  \max_{ j = k}^n B_j(k) = \max_{\substack{k \leq j \leq n \\ k + j \in 2\Z}} \prob{W_j = k}. 
\end{equation}
To finish the proof, note that the sequence $\prob{ W_j = k } = 2^{-j}\binom{j}{(k+j)/2}$, $j = k, k+2, k+4, \dots$ is unimodal with a peak at $j=k^2$, i.e.,
\[ \prob{W_{j-2} = k} \leq \prob{W_j = k}, \quad \text{if} \quad j \leq k^2,\]
and
\[ \prob{W_{j-2} = k} > \prob{W_j = k}, \quad \text{if} \quad j > k^2. \]
Indeed, elementary calculations yield that the inequality 
\[ 2^{-j+2} \binom{j-2}{(k+j)/2 - 1} \leq 2^{-j} \binom{j}{(k+j)/2}, \qquad j \geq k + 2, \]
is equivalent to the inequality $j \leq k^2$. \qquad\endproof

\section{Proofs based on results in extremal combinatorics}\label{sec3}

Let $[n]$ stand for the finite set $\{1, 2, \dots, n\}$. Consider a family $\mathcal{F}$ of subsets of $\left[n\right]$. We denote by $\left|\mathcal{F}\right|$ the cardinality of $\mathcal{F}$. The family $\mathcal{F}$ is called:
 
 \begin{romannum}
   \item \textbf{$k$-intersecting} if for all $A,B \in \mathcal{F}$ we have $\left|A\cap B\right|\geq k$.
 \item an \textbf{antichain} if for all $A,B \in \mathcal{F}$ we have $A \nsubseteq B$.
   \end{romannum}
\
 A well known result by Katona \cite{katona} (see also \cite[p. 98, Theorem 4]{bollobas}) gives the exact upper bound for a $k$-intersecting family.
 \begin{theorem}[Katona \cite{katona}] \label{thm3} If $k \geq 1$ and $\mathcal{F}$ is a $k$-intersecting family of subsets of $\left[n\right]$ then
\begin{equation}\label{Kat}
\left|\mathcal{F}\right|\leq \left\{ 
\begin{gathered}
  \sum_{j=t}^{n} \binom{n}{j}, \qquad\qquad\qquad  \text{\rm if}\;\; k+n = 2t, \hfill \\
  \sum_{j=t}^{n} \binom{n}{j}+\binom{n-1}{t-1}, \quad \text{\rm if}\;\; k+n = 2t - 1. \hfill \\
\end{gathered} \right.
\end{equation}
\end{theorem}
\noindent Notice that if $k + n = 2t$, then 
\begin{equation}\label{tail1}
  \sum_{j=t}^{n} \binom{n}{j} = 2^n \prob{W_n \geq k}. 
\end{equation}
If $k + n = 2t - 1$, then using the Pascal's identity $\binom n j = \binom{n-1} j + \binom {n-1}{j-1}$ we get
\begin{equation}\label{tail2}
  \sum_{j=t}^{n} \binom{n}{j}+\binom{n-1}{t-1} = 2 \sum_{j = t - 1}^{n-1} \binom {n-1}{j} = 2^n \prob{W_{n-1} \geq k}.
\end{equation}
The exact upper bound for the size of a $k$-intersecting antichain is given by the following result of Milner \cite{milner}.
\begin{theorem}[Milner \cite{milner}]\label{thm4}
If a family ${\cal F}$ of subsets of $[n]$ is a $k$-intersecting antichain, then 
  \begin{equation}\label{Mil}
    |{\cal F}| \leq 
	\binom{n}{t}, \qquad t = \left\lceil \frac {n+k} 2 \right\rceil. 
  \end{equation}
\end{theorem}
\noindent Note that we have 
\begin{equation}\label{layer1}
  \binom{n}{t} = 2^n \prob{W_n = k}, \quad \text{if} \quad n + k = 2t,
\end{equation}
and
\begin{equation}\label{layer2}
  \binom{n}{t} =  2^n \prob{W_{n} = k+1}, \quad \text{if} \quad n + k = 2t - 1.
\end{equation}

\begin{flushleft}
  By Lemma \ref{reduction} it is enough to prove Theorems \ref{thm1} and \ref{thm2} for the sums
\end{flushleft}
$$S_n = a_1 \eps_1 + \dots + a_n \eps_n,$$
where $0 \leq a_1, \dots , a_n \leq 1$. Denote as $A^{c}$ the complement of the set $A$.
For each $A \subset [n]$, write $s_A = \sum_{i \in A} a_i - \sum_{i \in A^c} a_i$. We define two families of sets:
$$\mathcal{F}_{\geq x} = \{A \subset [n] : s_A \geq x\}, \quad \text{and} \quad \mathcal{F}_{x} = \{A \subset [n] : s_A = x\}.$$
{\em Proof of Theorem \ref{thm1}}.  We have
  $$ \prob{S_n \geq x} = 2^{-n} |\mathcal{F}_{\geq x}|.$$
  Let $k = \lceil x \rceil$. Since $W_n$ takes only integer values, we have
  $$ \prob{W_n \geq k} = \prob{W_n \geq x} \qquad \text{and} \qquad \prob{W_{n-1} \geq k} = \prob{W_{n-1} \geq x}.$$
  Therefore, in the view of \eqref{Kat}, \eqref{tail1}, and \eqref{tail2}, it is enough to prove that $\mathcal{F}_{\geq x}$ is $k$-intersecting.
  Suppose that there are $A, B \in \mathcal{F}_{\geq x}$ such that $|A \cap B| \leq k - 1$. Writing $\sigma_A = \sum_{i \in A} a_i$, we have
  \begin{equation}\label{one}
    s_A = \sigma_A - \sigma_{A^{c}} =  (\sigma_{A \cap B} - \sigma_{A^c \cap B^{c}}) + (\sigma_{A\cap B^{c}} - \sigma_{A^{c}\cap B})                                                                                                                        \end{equation} 
  and
  \begin{equation}\label{two}
    s_B = \sigma_B - \sigma_{B^{c}} = (\sigma_{A \cap B} - \sigma_{A^{c}\cap B^{c}}) - (\sigma_{A \cap B^{c}} - \sigma_{A^{c}\cap B}). 
  \end{equation} 
  Since 
  $$\sigma_{A \cap B} - \sigma_{A^{c}\cap B^{c}} \leq \sigma_{A \cap B} \leq |A \cap B| \leq k -1 < x,$$
  from \eqref{one} and \eqref{two} we get
  $$\min\{s_A, s_B\} < x,$$ 
  which contradicts the fact $s_A, s_B \geq x$. \qquad \endproof

The following lemma implies Theorem \ref{thm2}. It also gives the optimal bound for $\prob{S_n = x}$ and thus improves Lemma \ref{apelsinas}. 
\begin{lemma}\label{lem1}
  Let $0 < a_1,\dots, a_n \leq 1$ be strictly positive numbers, $x>0$, $k = \lceil x \rceil$. Then
  \begin{equation*}
    \prob{S_n = x} \leq 
      \begin{cases}
        \prob{W_n = k}, \quad &\text{\rm if} \quad n + k \in 2\Z, \\
        \prob{W_n = k + 1}, \quad &\text{\rm if} \quad n + k \in 2\Z + 1.
      \end{cases}
  \end{equation*}
\end{lemma}
\begin{proof}
  We have
  $$ \prob{S_n = x} = 2^{-n} |\mathcal{F}_{ x}|.$$
  In the view of \eqref{Mil}, \eqref{layer1}, and \eqref{layer2}, it is enough to prove that $\mathcal{F}_{ x}$ is a $k$-intersecting antichain. To see that $\mathcal{F}_{ x}$ is $k$-intersecting it is enough to note that $\mathcal{F}_{ x} \subset \mathcal{F}_{\geq x}$.  To show that $\mathcal{F}_{ x}$ is an antichain is even easier. If $A, B \in \mathcal{F}_{ x}$ and $A \subsetneq B$, then $s_B - s_A$ = $2 \sum_{i \in B \backslash A} a_i > 0$, which contradicts the assumption that $s_B = s_A = x$.
\qquad\end{proof}

{\em Proof of Theorem \ref{thm2}}. Lemma \ref{lem1} gives 
  \[ \prob{S_n = x} \leq  \max_{ j = k}^n \prob{W_j = k + 1 - \mathbb{I}(k,j)},\]
  where again $\mathbb{I}(k,j)=\mathbb{I}\left\{k+j \in 2\Z \right\}$.  Note that if $k + j \in 2\Z$ we have
  \begin{align*} 
    \prob{W_j = k} &\geq  1 / 2  \prob{W_j = k} + 1 / 2 \prob{W_j = k + 2} \\
                   &=\prob{W_{j+1} = k + 1}, \qquad k > 0.
  \end{align*}
  Hence
  \[ \max_{ j = k}^n \prob{ W_j = k + 1- \mathbb{I}(k,j)  } = \max_{\substack{k \leq j \leq n \\ k + j \in 2\Z}} \prob{ W_j = k },\]
  the right-hand side being the same as the one of \eqref{bananas}. Therefore, repeating the argument following \eqref{bananas} we are done.
\qquad\endproof

\section{Extension to Lipschitz functions}
\label{sec4}
One can extend Theorem \ref{thm1} to odd Lipschitz functions taken of $n$ independent random variables. Consider the cube $C_n = [-1,1]^n$ with the $\ell^1$ metric $d$.
We say that a function $f : C_n \to \R$ is $K$-Lipschitz with $K > 0$ if 
\begin{equation}\label{Lip}
|f(x) - f(y)| \leq K d(x,y), \qquad x,y \in C_n.
\end{equation}
We say that a function $f : C_n \to \R$ is odd if $f(-x) = -f(x)$ for all $x \in C_n$. An example of an odd $1$-Lipschitz function is the function mapping a vector to the sum of its coordinates: 
$$f(x_1, \dots, x_n) = x_1 + \dots + x_n.$$
Note that the left-hand side of \eqref{eq1} can be written as $\prob{f(X_1,\dots,X_n) \geq x}$.

As in Theorems \ref{thm1} and \ref{thm2}, the crux of the proof is dealing with two-valued random variables. The optimal bound for a $k$-intersecting family is not sufficient for this case, therefore we use the following generalization of Theorem \ref{thm3} due to Kleitman \cite{kleitman} (see also \cite[p. 102]{bollobas}) which we state slightly reformulated for our convenience. Let us define the diameter of a set family $\mathcal F$ by $\operatorname{diam} \mathcal{F} = \max_{A,B \in \mathcal{F}} |A\bigtriangleup B|$.
\begin{theorem}[Kleitman \cite{kleitman}]\label{thm5}
  If $k \geq 1$ and $\mathcal{F}$ is a family of subsets of $\left[n\right]$ with $\operatorname{diam} \mathcal{F} \leq n - k$, then
\begin{equation}\label{Kle}
\left|\mathcal{F}\right|\leq \left\{ 
\begin{gathered}
  \sum_{j=t}^{n} \binom{n}{j}, \qquad\qquad\qquad  \text{\rm if}\;\; k+n = 2t, \hfill \\
  \sum_{j=t}^{n} \binom{n}{j}+\binom{n-1}{t-1}, \quad \text{\rm if}\;\; k+n = 2t - 1. \hfill \\
\end{gathered} \right.
\end{equation}
\end{theorem}

To see that Theorem \ref{thm5} implies Theorem \ref{thm3}, observe that $|A \cap B| \geq k$ implies $|A\bigtriangleup B| \leq n - k$.

\begin{theorem}
 Suppose that a function $f : C_n \rightarrow \R$ is $1$-Lipschitz and odd. Let $X_1, \dots, X_n$ be symmetric independent random variables such that $|X_i| \leq 1$. Then, for $x > 0$, we have that
\begin{equation}\label{orangutanas}
  \prob{f(X_1, \dots, X_n) \geq x} \leq 
    \begin{cases} 
      \prob{W_n \geq x}, \quad &\mathrm{if} \quad n + \lceil x \rceil \in 2 \Z, \\ 
      \prob{W_{n-1} \geq x}, \quad &\mathrm{if} \quad n + \lceil x \rceil \in 2 \Z + 1.
    \end{cases}
\end{equation}
\end{theorem}

\begin{proof}
  Applying Lemma \ref{reduction} with the function
  $$g(y_1, \dots, y_n) = \I \{ f(y_1, \dots, y_n) \geq x \}, $$
  we can see that it is enough to prove \eqref{orangutanas} with
  $$X_1 = a_1 \eps_1, \dots, X_n = a_n \eps_n$$
  for any $1$-Lipschitz odd function $f$. In fact, we can assume that $a_1 = \dots = a_n =1$, since the function
  $$(x_1, \dots, x_n) \mapsto f(a_1 x_1, \dots, a_n x_n)$$
  is clearly $1$-Lipschitz and odd.
 
  Given $A \subseteq [n]$, write $f_A$ for $f(2 \,\Ind_A(1) - 1, \dots, 2 \,\Ind_A(n) - 1)$, where $\Ind_A$ is the indicator function of the set $A$. Note that
  \begin{equation}\label{lip}
    |f_A - f_B| \leq 2|A \bigtriangleup B|
  \end{equation}
  by the Lipschitz property. Consider the family of finite sets
  $$ \mathcal{F} = \{ A \subseteq [n] : f_A \geq x \},$$
  so that $$\prob{f(\eps_1, \dots, \eps_n) \geq x} = 2^{-n} |\mathcal{F}|.$$
  Write $k = \lceil x \rceil$. Note that $W_{n-1}$ and $W_n$ take only integer values. Therefore by \eqref{tail1} and \eqref{tail2} we see that the right-hand side of \eqref{Kle} is equal, up to the power of two, to the right-hand side of \eqref{orangutanas}. Consequently, if $\operatorname{diam} \mathcal{F} \leq n - k$, then Theorem \ref{thm5}  implies \eqref{orangutanas}.  
  Therefore, it remains to check that for any $A, B \in \mathcal{F}$ we have $|A \bigtriangleup B| \leq n-k$. 
  
  Suppose that for some $A, B$ we have $f_A, f_B \geq x$ but $|A \bigtriangleup B| \geq n-k + 1$. Then
  $$|A \bigtriangleup B^{c}| = |(A \bigtriangleup B)^{c}| = n - |A \bigtriangleup B| \leq k - 1,$$
  and hence by \eqref{lip} we have
  \begin{equation}\label{statement}
  |f_A - f_{B^{c}}| \leq 2k - 2.
  \end{equation}                                                     
  On the other hand we have that $f_{B^{c}} \leq -x$, as $f$ is odd. Therefore
  $$f_A - f_{B^{c}} \geq 2x > 2k - 2,$$
  which contradicts \eqref{statement}.
\qquad\end{proof}

\section*{Acknowledgement}
We would like to thank Paul Balister for careful reading and valuable remarks, which improved the exposition.

\bibliographystyle{siam}

\end{document}